\documentclass[10pt]{amsart}

\usepackage{graphicx}
\usepackage{amssymb}
\usepackage[colorlinks=true,
                    linkcolor=blue,
                    urlcolor=blue,
                    citecolor=blue,
                    anchorcolor=blue]{hyperref}
\usepackage{mathtools}
\usepackage{amsmath,amssymb}

\theoremstyle{plain}
\newtheorem{theorem}{Theorem}[section]
\newtheorem{lemma}[theorem]{Lemma}
\newtheorem{proposition}[theorem]{Proposition}
\newtheorem{corollary}[theorem]{Corollary}

\theoremstyle{remark}
\newtheorem{remark}[theorem]{Remark}

\theoremstyle{definition}

\newtheorem*{main}{Main Theorem}

\def\Mod{{\rm Mod}}

\begin{document}

\title[Constructing Lefschetz Fibrations with Arbitrary Slope] 
{Constructing Lefschetz Fibrations with Arbitrary Slope}

\author[T{\"{u}}l\.{i}n Altun{\"{o}}z and Adalet \c{C}engel]
{T{\"{u}}l\.{i}n Altun{\"{o}}z and Adalet \c{C}engel}

\thanks{The second author was supported by Scientific and Technological Research Council of T\"urkiye (T\"UB\.{I}TAK) under Grant Number 124F502.}

\address{Faculty of Engineering, Ba\c{s}kent University, Ankara, Turkey}
\email{tulinaltunoz@baskent.edu.tr}
\address{Department of Mathematics, Bart\.{i}n University, Bart\.{i}n, Turkey}
\email{acengel@bartin.edu.tr}

\begin{abstract}
We prove that for any rational number $r\in (2,8)$, there exists a genus-$g$ Lefschetz fibration over the two-sphere with large enough genus-$g$ having the slope is $r$. 
\end{abstract}

\maketitle

\section{Introduction} 

 As an invariant the slope is derived from the geography problem of Lefschetz fibrations. It is originally defined for complex surfaces fibred over any genus-$k$ curve, but here we consider the ones over $\mathbb{S}^2$ both holomorphic and non-holomorphic. 

  Let $f: X \to \mathbb{S}^2$ be a nontrivial relatively minimal genus-$g$ Lefschetz fibration with $X$ being a closed oriented smooth 4-manifold. The \emph{slope} $\lambda_f$ of this fibration is the rational number  defined by
  
\[
\lambda_f = \frac{K_f^2(X)}{\chi_f(X)},
\]
where $ K_f^2(X) = c_1^2(X) + 8(g - 1) \quad \text{and} \quad \chi_f(X) = \chi_h(X) + (g - 1)$. Moreover for an almost complex closed 4-manifold $X$,
\[
c_1^2(X) = 3\sigma(X) + 2e(X)\quad \text{and} \quad \chi_h(X)=\frac{\sigma(X)+e(X)}{4},
\]
where $e(X)$ and $\sigma(X)$ are for the Euler characteristic and signature of $X$, respectively. Note that when $f$ is not a holomorphic bundle, $\chi_f(X) \neq 0 $.

Xiao \cite{xiao} proved that for holomorphic fibrations $4-4/g \leq \lambda_f \leq 12$ and called the former inequality as the \textit{the slope inequality}. Due to the results from \cite{LowslopeLefschetz,monden, monden2,smith,stipsicz}, it was shown that for Lefschetz fibrations over the two-sphere, $ K_f^2(X)\geq 4g-4$ and $\chi_f(X)\geq 1 $ so that $\lambda_f > 0 $. Ozbagci \cite{ozbagci} proved that $ c_1^2(X) \leq 10\chi_h(X)+2g-2 $ which implies that $ \lambda_f \leq 10$. It was conjectured by Hain \cite{amaros, endonagami} that for $g \geq 2$ all relatively minimal genus-$g$ Lefschetz fibrations over two-sphere satisfy the slope inequality. With these results, we have $4-4/g \leq \lambda_f \leq 10$. However, in \cite{LowslopeLefschetz, monden,monden3} examples of Lefschetz fibrations are produced which do not satisfy the slope inequality $4-4/g$. Since if a Lefschetz fibration is holomorphic then it satisfies the slope inequality, the constructed examples are non-holomorphic. Moreover in \cite{monden3} the constructed Lefschetz fibrations have $-1$-sections which implies they are fiber sum indecomposable. For the rest it is not known whether they are fiber sum indecomposable.  On the other hand there are other examples of Lefschetz fibrations which violate lower bounds of the slope for non-hyperelliptic holomorphic fibrations of genus $3,4$ and $5$ \cite{endonagami}. 

Basically there are two topological invariants which affect the slope: $e(X)$ and  $\sigma(X)$. It was conjectured by Gompf  \cite{Gompf} that every minimal symplectic $4$-manifold which is not diffeomorphic to a ruled surface has nonnegative Euler characteristic. Since every closed symplectic $4$-manifold after blow-ups has Lefschetz fibration structure, it would be true for them as well. For the the signature, results are less transparent. Most of the Lefschetz fibrations with negative signature come from algebraic geometry and over all it is not known that whether every negative value can be attained as a signature by some Lefshetz fibration over the two-sphere. Besides Baykur and Hamada \cite{baykurhamada} constructed genus-$g$ Lefschetz fibrations over the two-sphere for every $g=8m+1\geq 9(m\geq 1 )$ which are the first examples with nonnegative signature. Related with this, it can be deduced $\lambda_f > 8$ if and only if  $\sigma(X)>0$.

The following theorem is fundamental for our results.

\begin{theorem}\cite{LowslopeLefschetz}.
Assume $g \geq 2$. For each genus-$g$ Lefschetz fibration, there exists another genus-$g$ Lefschetz fibration with slope less(greater respectively) than $\lambda_f$. Moreover there is no genus-$g$ Lefschetz fibration whose slope is equal to the infimum(supremum respectively) of slopes of all genus-$g$ Lefschetz fibrations.
\end{theorem}

Let the functions $m_\lambda(g)$ and $M_\lambda(g)$ are the infimum and supremum of the set of all $ \lambda_{f}$'s where $ g\geq 2$. It is clear that $m_\lambda(g)> 0$ for all $g$ and $m_\lambda(g)=2$ if only $g=2$. The equality can be shown by using the signature formula of  genus-$2$ hyperelliptic Lefschetz fibrations. It was shown that for each $g\geq 3$, there exists a genus-$g$ Lefschetz fibration whose slope is greater than $2$ but arbitrarily close to $2$ \cite{LowslopeLefschetz}. Thus $m_\lambda(g) \geq 2$. To the best of our knowledge for $g\geq 3$, there is no example having slope exactly $m_\lambda(g)=2$.  Since $\lambda_f \leq 10$, then $M_\lambda(g)\leq 10$. Here our constructed Lefschetz fibrations have negative signatures so by the above result $M_\lambda(g) <8$. In fact, since $\lambda_f = 8$ if and only if the signature is zero, a Lefschetz fibration has slope exactly $8$ precisely when its signature is zero. Therefore, while the lower bound $m_\lambda(g) = 2$, which is attained for $g = 2$, there is no Lefschetz fibration of genus $g \geq 2$ with negative signature and slope $8$.

In this paper, we construct genus-$g$ Lefschetz fibrations over the two-sphere whose slopes arbitrarily close to $8$ by employing an iterative process containing two key tools: the Matsumoto relation and the generalized star relation. The method involves repeatedly taking fiber sums of copies of a previous fibration and then performing a generalized star relation substitution, which alters the topological invariants of total space in a way that increases the slope. This construction gives genus-$g$ Lefschetz fibrations with large enough $g$ over the two-sphere whose slopes form a sequence converging to $8$. Using this family of Lefschetz fibrations and another family having low-slopes obtained in~\cite{LowslopeLefschetz} as building blocks we prove the following result:
\begin{main}\label{main}
For any rational number $r\in (2,8)$ there exists a genus-$g$ Lefschetz fibration over the two-sphere with large enough genus-$g$ having the slope equals to $r$.
\end{main}
\noindent Since the slope invariant is a rational number, the interval $(2,8)$ is filled for genus $g$ large enough.


\section{Preliminaries}
This section outlines the necessary background and gives the notation used throughout the paper.

Let $\Sigma_{g}^b$ be a compact, connected, oriented surface of genus $g$ with $b \geq 0$ boundary components. The mapping class group of this surface, denoted $\Mod(\Sigma_{g}^b)$, is defined as the group whose elements are isotopy classes of orientation-preserving self-diffeomorphisms of $\Sigma_{g}^b$, where every diffeomorphism is required to fix every point on the boundary. Also, the isotopies are also constrained to be the identity on the boundary. In the case of a surface with no boundary ($b=0$), we omit $b$ from the notation and write $\Sigma_{g}$ and $\Mod(\Sigma_{g})$.

Throughout this paper, a lowercase letter denotes a simple closed curve, and the corresponding capital letter denotes its positive (right-handed) Dehn twist. Our notation follows the functional order: for two diffeomorphisms $f$ and $g$, the composition $fg$ means $g$ first, then $f$. We denote the conjugation of $g$ by $f$ as $g^f$. All diffeomorphisms and curves are to be understood up to isotopy.

The mapping class group $\Mod(\Sigma_g^b)$ satisfies the following fundamental relations:

\begin{itemize}
    \item \textbf{Conjugation:} For any $f \in \Mod(\Sigma_{g}^b)$ and simple closed curves $x$ and $y$ with $f(x) = y$, the corresponding Dehn twists are conjugate:
        \[
        X^f = Y.
        \]
    \item \textbf{Commutativity:} If $x$ and $y$ are disjoint simple closed curves, then their Dehn twists commute:
        \[
        XY = YX.
        \]
    \item \textbf{Braid Relation:} If $x$ and $y$ intersect transversely at a single point, then their Dehn twists satisfy the braid relation:
        \[
        XYX= YXY.
        \]

\item \textbf{Chain Relation:} Let $(c_1, \dots, c_{2g+1})$ be sequence of simple closed curves on a surface $\Sigma_g^2$ as depicted in Figure~\ref{GSR} (here, we cap off the boundary component $\delta_1$). Then the following relation holds in $\Mod(\Sigma_g^2)$, called the \textit{odd chain relation}:
\[
(C_1 C_2 \cdots C_{2g} C_{2g+1})^{2g+2} = \Delta_2\Delta_3.
\]
The corresponding element
\[
\mathcal{C}_{2g+1} = (C_1 C_2 \cdots C_{2g} C_{2g+1})^{2g+2}\Delta_3^{-1}\Delta_2^{-1}
\]
is called the \textit{odd chain relator}.

Capping off the boundary components $\delta_1$ and $\delta_2$, the sequence of simple closed curves $(c_1, \dots, c_{2g})$ satisfies the following relation in $\Mod(\Sigma_g^1)$:
\[
(C_1 C_2 \cdots C_{2g} C_{2g})^{4g+2}=\Delta_3.
\]
The corresponding element
\[
\mathcal{C}_{2g} = (C_1 C_2 \cdots C_{2g})^{4g+2}\Delta_3^{-1}\
\]
is called the \textit{even chain relator}.


\item \textbf{Generalized Matsumoto Relation:} Let $b_0,b_1,\ldots,b_g, a,b$ and $c$ be simple closed curves on $\Sigma_{g}$ as shown in Figure \ref{GMR}. Then the following relation, called \textit{generalized Matsumoto relation}, holds in $\Mod(\Sigma_g)$:

\[
W_g=\left\{\begin{array}{lll}
(B_0B_1B_2\cdots B_gC)^{2}=1& \textrm{if} & g=2k,\\
(B_0B_1B_2\cdots B_g A^{2}B^{2})^{2}=1& \textrm{if} & g=2k+1.\\
\end{array}\right.
\]
 \begin{figure}[h]
\begin{center}
\scalebox{0.35}{\includegraphics{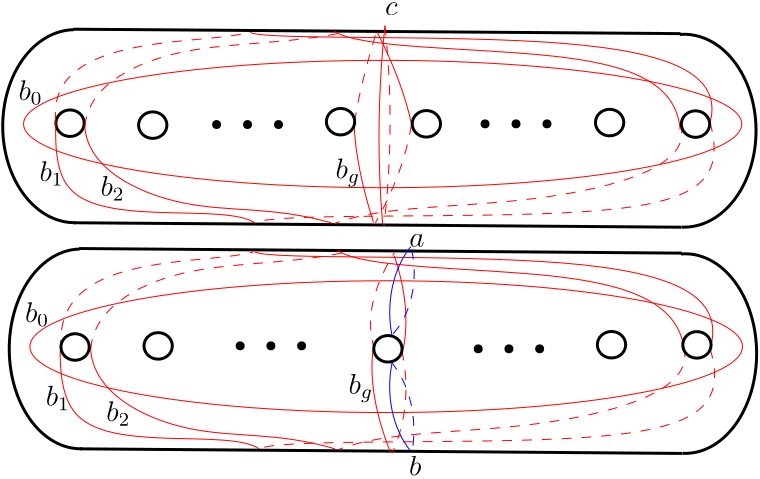}}
\caption{The curves in generalized Matsumoto relation.}
\label{GMR}
\end{center}
\end{figure}
Matsumoto~\cite{matsumoto2} proved that the relation $W_g=1$ in $\Mod(\Sigma_g)$ for $g=2$. Later, this result extended to $g\geq 3$ by the independent work of Korkmaz~\cite{korkmaz2001} and Cadavid~\cite{cadavid}.

    \item \textbf{Generalized Star Relation:} Let $\delta_1$, $\delta_2$ and $\delta_3$ be boundary parallel curves of a surface $\Sigma_g^3$, and let $c_1, c_2, \ldots, c_{2g+2}$ be the simple closed curves as shown in Figure~\ref{GSR}. Then the following relation, known as the \textit{generalized star relation}~\cite{akhmedov2020}, holds in $\Mod(\Sigma_g^3)$:
    \[
(C_{2g+2}C_{2g+1}C_{2g}\cdots C_{2}C_{1})^{2g+1}=\Delta_1 \Delta_2^g \Delta_3.
    \]
 The corresponding \textit{generalized star relator} is as follows 
        \[
S_g=(C_{2g+2}C_{2g+1}C_{2g}\cdots C_{2}C_{1})^{2g+1}\Delta_3^{-1} \Delta_2^{-g} \Delta_1^{-1}.
    \]
  
    \begin{figure}[h]
\begin{center}
\scalebox{0.27}{\includegraphics{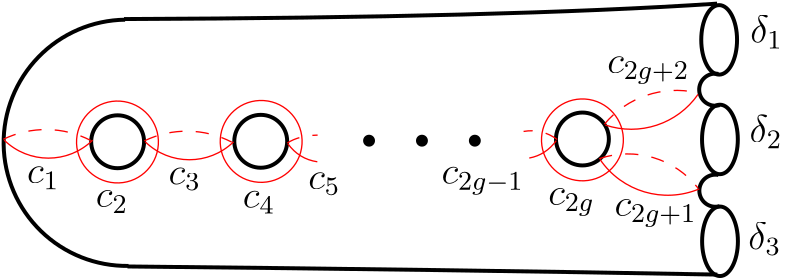}}
\caption{The curves in generalized star relation.}
\label{GSR}
\end{center}
\end{figure}

\noindent When $g=1$, it is a well known relation in $\Mod(\Sigma_1^3)$, called \textit{star relation}\cite{endonagami}.
\end{itemize}

\subsection{Lefschetz fibrations}

Let \( X \) be a closed oriented smooth 4-manifold and let $\mathbb{S}^2$ denote the two-sphere. A smooth map \( f: X \to \mathbb{S}^2 \) is called a \textit{genus-\(g\) Lefschetz fibration} on $X$ having finitely many critical values \( b_1, \ldots, b_n \) in such a way that there exists a unique critical point \( p_i \) in the singular fiber \( f^{-1}(b_i) \) such that around each \( p_i \) and \( b_i \) there are orientation--preserving local complex coordinate charts on which  \( f \) is given by \( f(z_1, z_2) = z_1^2z_2^2 \). (In general, the base space of a Lefschetz fibration can be any closed orientable surface; however, we restrict ourselves to the case where it is the two-sphere $\mathbb{S}^2$.) We assumed that \( f \) is relatively minimal. This means no fiber contains an exceptional sphere, i.e., a smoothly embedded sphere with self-intersection number $(-1)$. Each regular fiber of 
$f$ is a closed connected orientable genus-$g$ surface $\Sigma_g$; this positive integer $g$ is referring to as the genus of the Lefschetz fibration.

Each singular fiber of the Lefschetz fibration $(X,f)$ results from pinching  a specific simple closed curve, called a \textit{vanishing cycle} on a regular fiber to a single point. The vanishing cycle determines the diffeomorphism type of a regular neighborhood of the corresponding singular fiber.

For a genus-$g$ Lefschetz fibration $(X,f)$, let $b_0$ be a fixed regular value and let us identify the regular fiber $f^{-1}(b_0)$ with the genus-$g$ surface $\Sigma_g$. The Lefschetz fibration $(X,f)$ with vanishing cycles $\{ a_i\}$ is then encoded by a \textit{monodromy factorization}: a relation
\[
A_1 A_2 \cdots A_n = 1 \quad \text{in } \Mod(\Sigma_g).
\]

The monodromy factorization of a Lefschetz fibration is not unique, ; however, any two representations of which are related by two fundamental operations:

\textit{Hurwitz Moves:} This move swaps the order of Dehn twists $A_i$ and $A_{i+1}$ in the factorization, while conjugating one of them to preserve the overall product, which is defined as:
    \[
    A_1A_2 \cdots A_i A_{i+1} \cdots A_n \sim A_1A_2 \cdots \left(A_{i+1}\right)^{A_{i}}A_i\cdots A_n,
    \]
    where \(\left(A_{i+1}\right)^{A_{i}} = A_iA_{i+1} A_i^{-1}\) denotes the conjugation.
    
   \textit{Global Conjugation:} The entire factorization can be conjugated by a diffeomorphism \(\phi\) in $\Mod({\Sigma_g})$:
    \[
    A_1A_2 \cdots A_n \sim  C_1C_2 \cdots C_n,
    \]
    where \(\phi(a_i)=c_i\).

Let $f_i: X_i \to \mathbb{S}^2$ be a genus-$g$ Lefschetz fibration with monodromy factorization $W_i$, for each $i=1,2$. From each  $X_i$, we remove a  tubular neighborhood $\upsilon \Sigma_g$ of one chosen regular fiber 
$\Sigma_g$. Then we glue  the resulting manifolds along their boundaries using a specific diffeomorphism $\phi$, which yields a genus $g$ Lefschetz fibration
\[
f_1\sharp f_2:X_1 \sharp X_2 \to \mathbb{S}^2
\]
with monodromy factorization $W_1W_2^{\phi}$. This operation is called the \textit{(twisted) fiber sum} of $X_1$ and $X_2$.

For a genus-$g$ Lefschetz fibration $f:X\to \mathbb{S}^2$ with $n$ singular fibers, the Euler characteristic is given by $e(X)=4-4g+n$. Another topological invariant is the signature, $\sigma(X)$, which can be computed using various methods~\cite{Cengel2020,endo,endonagami,matsumoto2,ozbagci,smith}. 

Let $x_1, x_2, \dotsc, x_m$ and $y_1, y_2, \dotsc, y_n$ be simple closed curves on $\Sigma_g$ whose positive Dehn twists satisfy the relation
\[
X_1 X_2 \cdots X_m = Y_1 Y_2 \cdots Y_n.
\]

Given a positive relator $W = U \cdot X_1 X_2 \cdots X_m \cdot V$, we can form a new positive relator
\[
W' = U \cdot Y_1 Y_2 \cdots Y_n \cdot V
\]
by replacing the subword $X_1 X_2 \cdots X_m$ in $W$ with $Y_1 Y_2 \cdots Y_m$. In this case, letting $R = Y_1 Y_2 \cdots Y_n X_m^{-1} \cdots X_2^{-1} X_1^{-1}$, we say that $W'$ is obtained from $W$ by an \textit{$R$-substitution}. Thus, a new Lefschetz fibration with monodromy $W'$ can be constructed from a fibration with monodromy $W$ by performing an $R$-substitution.

Suppose that $1\leq h \leq g-2$. One can embed the generalized star relator $S_h$ into $\Mod(\Sigma_g)$ (in this embedding, the curves in the relation can be as in Figure~\ref{GSRH}). The following proposition describes the effect of a $S_h$-substitution on the slope of a Lefschetz fibration.

\begin{proposition}~\cite[Lemma $15$]{akhmedov2020}\label{prop1}
Let $1\leq h \leq g-2$ and let genus-$g$ Lefschetz fibration $f_2:X_{2}\to \mathbb{S}^2$ is obtained from the genus-$g$ Lefschetz fibration $f_1:X_{1}\to \mathbb{S}^2$  by an $S_h$-substitution. Then the topological invariants of $X_2$ are as follows:
\[
\quad e(X_{2}) = e(X_{1}) -4h^2-5h \quad \text{and} \quad \sigma(X_{2}) = \sigma(X_{1})+2h^2 +3h.  
\]
Thus,
\[
\chi_{f_{2}} = \chi_{f_{1}}-\dfrac{h^2+h}{2} \quad \text{and} \quad K_{f_{2}}^2 = K_{f_{1}}^2 - 2h^2-h.
\]
\end{proposition}
Before proceeding, we state how the fiber sum construction affects the slope of the total spaces. This is given by the following result:
\begin{lemma}\cite{LowslopeLefschetz}\label{lemmafibersum}
Let $f_i: X \to S^2$ be Lefschetz fibration of genus $g \geq 2$ for each $i=1,2$. For any fiber sum $f = f_1 \#f_2$, the following  relations hold for its invariants:
\[
K^2_f = K^2_{f_1} + K^2_{f_2} \quad \text{and} \quad \chi_f = \chi_{f_1} + \chi_{f_2}.
\]
It follows immediately that the slope of the resulting fibration is
\[
\lambda_f = \frac{K^2_{f_1} + K^2_{f_2}}{\chi_{f_1} + \chi_{f_2}}.
\]
\end{lemma}

\section{Lefschetz fibrations with arbitrary slope}

\subsection{Low--slope Lefschetz fibrations}
Xiao \cite{xiao} proved that the lower bound of $\lambda_f$ for all relatively minimal holomorphic fibrations is $ 4-4/g$. Due to the recent results this bound for Lefschetz fibrations is less.


Monden \cite{monden} was first by constructing examples of Lefschetz fibrations with slope smaller than $ 4-4/g$. For the construction first hyperelliptic Lefschetz fibrations are considered whose slope on the bound $ 4-4/g$. By degenerating hyperellipcity and taking twisted fiber sums new Lefschetz fibrations are produced. Using inverse Lantern substitutions gives smaller slopes. Then applying the same procedure with the previous Lefschetz fibrations yields much smaller slopes. However for large $g$, the slopes getting close to $4$. The second author and Korkmaz \cite{LowslopeLefschetz} constructed a sequence of Lefschetz fibrations having slope less than the existing ones. To be more precise for each $g \geq 3$, there are Lefschetz fibrations whose slopes are greater than $2$ but arbitrarily close to $2$. The idea of the construction is as follows. Let $g\geq 3$ and $f_0: X_0 \to \mathbb{S}^2$ be a genus-$g $ Lefschetz fibration with any slope. Due to the result of Smith \cite{amaros,smith}, the monodromy of $f$ contains a positive Dehn twist $D^r$ about a nonseparating curve with order $r\geq 1$. Then conjugation of the monodromy by appropriate two diffeomorphisms separately and followed by a twisted fiber sum gives a new Lefschetz fibration. In the resulting monodromy, there exist two distinct curves bounding a subsurface of genus-$h$ involving a chain of lenght $(2h+1)$ curves. Denote the boundary paralel curves by $e_{h+1}$ and $f_{h+1}$. For any value of $1\leq h \leq g-1$, taking out $E_{h+1}F_{h+1}$ and inserting $(C_1 C_2  \cdots C_{2h} C_{2h+1})^{2h+2}$ decreases the initial slope. Iteratively the previous Lefschetz fibration $f_n: X_n \to \mathbb{S}^2$ is used to obtain smaller slope. The sequence of slopes converges to $4h/(h+1)$. The smallest slope is obtained when the chain is shortest that is $h=1$ and it is independent of the initial slope. Note that at the same time one can obtain different pair of boundary curves with different orders. Specifically when $f_0: X_0 \to \mathbb{S}^2$ is the hyperelliptic Lefschetz fibration with nonseparating vanishing cycles, there is a sequence of $f_n$'s such that $2 < \lambda_{f_n} < 2+\frac{4g-8}{2^n}.$ The corresponding total spaces are minimal and simply-connected. For more details one can check \cite{LowslopeLefschetz}.

According to the Baykur and Hamada for all genus-$g$ Lefschetz fibrations over $\mathbb{S}^2$, the slope is at least $2$. Depending on this, it was conjectured by Cengel and Korkmaz for $g \geq 2$, the lower bound for the slope of all Lefschetz fibrations over the two-sphere is $2$ \cite{LowslopeLefschetz}. 

\begin{remark} In ~\cite{miyachishiga} Miyachi and Shiga show that there are genus-$g$ Lefschetz fibrations over positive genus surfaces with negative slope.
\end{remark}

\subsection{High--slope Lefschetz fibrations}
For a genus-$g$ Lefschetz fibration over $\mathbb{S}^2$, the inequality $\lambda_f < 8$ is equivalent to $\sigma(X) < 0$.  In this section, we aim to construct genus-$g$ Lefschetz fibrations over $\mathbb{S}^2$ having negative signature and a slope arbitrarily close to this supremum $8$.

 We construct high-slope Lefschetz fibrations using the Matsumoto relation \(W_g\) as a building block. It is known that the total space \(X_{W_g}\) of the Lefschetz fibration \(f_{W_g}: X_{W_g} \to \mathbb{S}^2\) has the following invariants \cite{endonagami}:
\[
\sigma(X_{W_g}) = \begin{cases}
-4 & \text{if } g \text{ is even}, \\
-8 & \text{if } g \text{ is odd},
\end{cases}
\quad \text{and} \quad
e(X_{W_g}) = \begin{cases}
8 - 2g & \text{if } g \text{ is even}, \\
14 - 2g & \text{if } g \text{ is odd}.
\end{cases}
\]
From these, one can compute the following:
\[
K_{f_{W_g}}^2 = 4g - 4
\quad \text{and} \quad
\chi_{f_{W_g}} = \begin{cases}
g/2 & \text{if } g \text{ is even}, \\
(1+g)/2 & \text{if } g \text{ is odd}.
\end{cases}
\]
Therefore, the slope of the fibration is given by
\[
\lambda_{f_{W_g}} = \frac{K_{f_{W_g}}^2}{\chi_{f_{W_g}}} = \begin{cases}
8(g-1)/g & \text{if } g \text{ is even}, \\[1.5em]
8(g-1)/(g+1) & \text{if } g \text{ is odd}.
\end{cases}
\]
\begin{theorem}\label{thmhigh}
 For $g\geq 3$ and $1\leq h\leq g-2$, there exists a Lefschetz fibration $f:X \to \mathbb{S}^2$ with slope \[
\lambda_{f}=\begin{cases}
8-\dfrac{2(14h^2+21h+8)}{(h+1)(4gh+2g-h)} & \text{if } g \text{ is even}, \\[1.5em]
8-\dfrac{2(30h^2+45h+16)}{(h+1)(4gh+2g+3h+2)} & \text{if } g \text{ is odd}.
\end{cases}
\]  
\end{theorem}
\begin{proof}
   First, let us embed the generalized star relation into $\Mod(\Sigma_g)$ as in Figure~\ref{GSRH}.

 \begin{figure}[h]
\begin{center}
\scalebox{0.3}{\includegraphics{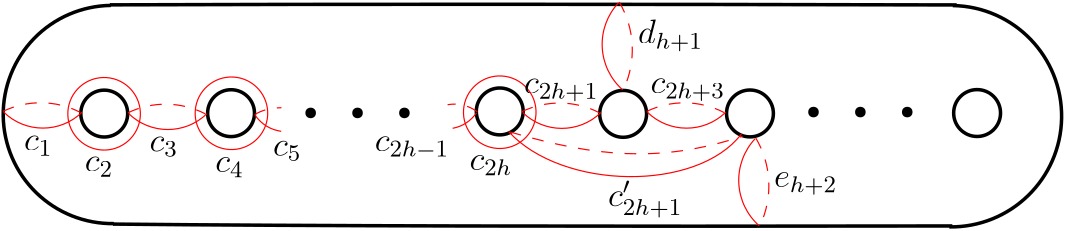}}
\caption{The curves in generalized star relation $S_h$ on $\Sigma_g$.}
\label{GSRH}
\end{center}
\end{figure}
Then the following relator exists in $\Mod(\Sigma_g)$ for all $1 \leq h \leq g-2$:
\begin{equation}
S_h=(C_{2h+1}'C_{2h+1}C_{2h} \cdots C_{2}C_{1})^{2h+1}E_{h+2}^{-1} C_{2h+3}^{-h} D_{h+1}^{-1}\label{eqn1}.    
\end{equation}
Since the curves $b_0$, $c_i$ for $1 \leq i \leq 2h+1$, and $c_{2h+1}'$ are nonseparating, there exist diffeomorphisms $\varphi_1, \ldots, \varphi_{2h+1}$ with $\varphi_i(b_0) = c_i$, and a diffeomorphism $\varphi_{2h+2}$ with $\varphi_{2h+2}(b_0) = c_{2h+1}'$. Thus, we have the following relator $W_g^{\varphi_i}=1$ in $\Mod({\Sigma_g})$ for $1 \leq i \leq 2h+2$:
\begin{align*}
 W_g^{\varphi_i} &= 
\begin{cases}
(B_0^{\varphi_i}B_1^{\varphi_i}\cdots B_g^{\varphi_i}C^{\varphi_i})^{2} & \text{if } g \text{ is even}, \\
(B_0^{\varphi_i}B_1^{\varphi_i}\cdots B_g^{\varphi_i}(A^2)^{\varphi_i}(B^2)^{\varphi_i})^{2}  & \text{if } g \text{ is odd},
\end{cases} 
\end{align*}
Observe that the Dehn twist $C_i$ appears in each factorization $ W_g^{\varphi_i}$.

Consider the relator $W:=W_g^{\varphi_{2h+2}}W_g^{\varphi_{2h+1}}\cdots W_g^{\varphi_{2}}W_g^{\varphi_{1}}$. By applying simultaneous conjugations to the relator $W$, it can be rewritten as follows:
\[
W = (C_{2h+1}' C_{2h+1} C_{2h} \cdots C_2 C_1)V,
\]
where $V$ is a word containing positive Dehn twists. Furthermore applying fiber sums and then via simultaneous conjugations, the relation \(W^{2h+1} = 1\) in \(\Mod(\Sigma_g)\) yields a new relator \(W'\), which can be expressed as follows:
\[
W'=W^{2h+1} = (C_{2h+1}' C_{2h+1} C_{2h} \cdots C_2 C_1)^{2h+1} V',
\]
where \(V'\) is a product of $(2h+1)(2h+2)(2g+3)$ positive Dehn twists if $g$ is even and of $(2h+1)(2h+2)(2g+9)$ positive Dehn twists if $g$ is odd. Since $f_{W'}:X_{W'}\to \mathbb{S}^2$ is $(2h+1)(2h+2)$-fiber sums of the Matsumoto fibration $f_{W_g}:X_{W_g}\to \mathbb{S}^2$, we get
\[
K_{f_{W'}}^2 = (2h+1)(2h+2)(4g - 4),
\]
and  

\[
\chi_{f_{W'}} = 
\begin{cases}
(2h+1)(2h+2)(g/2) & \text{if } g \text{ is even}, \\
(2h+1)(2h+2)(1+g)/2 & \text{if } g \text{ is odd}.
\end{cases}
\]
Applying $S_h$ substitution given in (\ref{eqn1}) to the relator $W'$, we get the following relator
\[
P_g=D_{h+1} C_{2h+3}^{h}E_{h+2}V'. 
\]
It follows from  Proposition~\ref{prop1}, the genus-$g$ Lefschetz fibration $f_{P_g}:X_{P_g}\to \mathbb{S}^2$ satisfies
\[
K_{f_{P_g}}^2 = (2h+1)(2h+2)(4g - 4)-2h^2-h,
\]
and 
\[
\chi_{f_{P_g}} = 
\begin{cases}
(2h+1)(2h+2)(g/2)-(h^2+h)/2 & \text{if } g \text{ is even}, \\
(2h+1)(2h+2)(1+g)/2-(h^2+h)/2& \text{if } g \text{ is odd}.
\end{cases}
\]
Therefore, the required slope follows immediately from the definition $\lambda_{f_{P_g}} = \frac{K_{f_{P_g}}^2}{\chi_{f_{P_g}}}$.


\end{proof}


\begin{corollary}\label{corhigh}
  For $g\geq 3$ and $1\leq h\leq g-2$, there exists a Lefschetz fibration $f:X \to \mathbb{S}^2$ whose slope converges to $$\frac{2\big[(16g - 18) h^2 + (24g - 25) h + 4(g - 1)\big]}
{(4 c_g - 1) h^2 + (6 c_g - 1) h + c_g},
$$
where 
$c_g = 
\begin{cases}
g, & \text{if $g$ is even},\\[1mm]
g+1, & \text{if $g$ is odd}.
\end{cases}$
\end{corollary}
\begin{proof}
Following the notation from the proof of Theorem~\ref{thmhigh}, let $P_{g,1}$ denote the factorization of the positive Dehn twists $P_g$. We now employ a similar argument used in that proof.

Since the monodromy $P_{g,1}$ contains at least one nonseparating simple closed curve (for example, $d_{h+1}$), one can find diffeomorphisms $\varphi_{i,2}$ (for $i=1,2, \ldots, 2h+1$) with $\varphi_{i,2}(d_{h+1}) = c_i$, and a diffeomorphism $\varphi_{2h+2,2}$ with $\varphi_{2h+2,2}(d_{h+1}) = c_{2h+1}'$.

Consider the following relator in $\Mod({\Sigma_g})$:
$$W_{P_{g,2}}:=P_{g,1}^{\varphi_{2h+2,2}}P_{g,1}^{\varphi_{2h+1,2}}\cdots P_{g,1}^{\varphi_{2,2}}P_{g,1}^{\varphi_{1,2}}.$$
This can be rewritten by applying conjugation relations simultaneously. The result of this process is the equivalent expression:
\[
W_{P_{g,2}}=(C_{2h+1}' C_{2h+1} C_{2h} \cdots C_2 C_1)V_{P_{g,2}},
\]
where $V_{P_{g,2}}$ is a product of positive Dehn twists. From these we derive a new relator as follows.
\[
W_{P_{g,2}}':=W_{P_{g,2}}^{2h+1}=(C_{2h+1}' C_{2h+1} C_{2h} \cdots C_2 C_1)^{2h+1}V_{P_{g,2}}',
\]
by again using conjugation relations repeatedly, where $V_{P_{g,2}}'$ is another word containing only positive Dehn twists. Now, applying $S_h$ substitution to the relator $W_{P_{g,2}}'$, one can obtain the following relator 
\[
P_{g,2}:=D_{h+1} C_{2h+3}^{h}E_{h+2}V_{P_{g,2}}'.
\]

It follows from Lemma~\ref{lemmafibersum} and Proposition~\ref{prop1} that the genus-$g$ Lefschetz fibration $f_{P_{g,2}}:X_{P_{g,2}}\to \mathbb{S}^2$ satisfies
\[
K_{f_{P_{g,2}}}^2 = (2h+1)(2h+2)K_{f_{P_{g,1}}}^2-2h^2-h,
\]
and 
\[
\chi_{f_{P_{g,2}}} =(2h+1)(2h+2)\chi_{f_{P_{g,1}}}-(h^2+h)/2.
\]
By iterating this construction process, we get a family of genus-$g$ Lefschetz fibrations$f_{P_{g,m}}:X_{P_{g,m}}\to \mathbb{S}^2$ (for $m\geq 1$) having 
\[
K_{f_{P_{g,m}}}^2 = \big((2h+1)(2h+2)\big)^m(4g-4)-\sum_{i=0}^{m-1} \big((2h+1)(2h+2)\big)^{i}(2h^2+h),
\]
and 
\begin{align*}
\chi_{f_{P_{g,m}}} = \frac{1}{2}\left[\big((2h+1)(2h+2)\big)^m c_g - (h^2+h)\sum_{i=0}^{m-1} \big((2h+1)(2h+2)\big)^{i}\right],
\end{align*} and so, the sequence $\lambda_{f_{P_{g,m}}}=\frac{K_{f_{P_{g,m}}}^2}{\chi_{f_{P_{g,m}}}}$ has the limit 
\[
\lim_{m \to \infty} \frac{K_{f_{P_{g,m}}}^2}{\chi_{f_{P_{g,m}}}}
=
\frac{2\big[(16g - 18) h^2 + (24g - 25) h + 4(g - 1)\big]}
{(4 c_g - 1) h^2 + (6 c_g - 1) h + c_g},
\]
where $c_g = \begin{cases} g & \text{if } g \text{ is even} \\ g+1 & \text{if } g \text{ is odd} \end{cases}$, which is the desired conclusion.
\end{proof}

The limit of the slope obtained in Corollary~\ref{corhigh} attains its maximum value when $h$ equals either 
$\lfloor h_{\max} \rfloor$ or $\lceil h_{\max} \rceil$, where
\[
h_{\max} = 
\begin{cases}
\displaystyle \frac{2(g-2) + \sqrt{4g^2 + 5g - 12}}{7}, & \text{if } g \text{ is even}, \\[5mm]
\displaystyle \frac{2(g-3) + \sqrt{4g^2 + 21g - 39}}{15}, & \text{if } g \text{ is odd}.
\end{cases}
\]
Since $h_{\max}$ is generally not an integer, the maximum is achieved at one of these two integer values.
Moreover, as $g$ increases, the corresponding maximal slope obtained in Corollary~\ref{corhigh} increases monotonically and converges to~$8$. 
Hence, the \emph{supremum} of this slope is $8$. 

In the following table, we list the approximated maximum slopes of Lefschetz fibrations constructed in Corollary~\ref{corhigh} for selected genera to observe how their slopes increase.
\begin{table}[htbp]
\centering
\begin{tabular}{|c|c|c|}
\hline
\textbf{Genus \( g \)} & \textbf{\( h_{\max} \)} & \textbf{Maximum slope} \\ \hline
3     & 1   & 4.0476 \\ \hline
10   & 5   & 7.2828 \\ \hline
15   & 4   & 7.0484 \\ \hline
20   & 11  & 7.6456 \\ \hline
25    & 7   & 7.4185 \\ \hline
30   & 17  & 7.7649 \\ \hline
50   & 28  & 7.8601 \\ \hline
100  & 57  & 7.9539 \\ \hline
200  & 115 & 7.9776 \\ \hline
1000 & 280 & 7.9950 \\ \hline
\end{tabular}
\caption{Maximum slopes of constructed Lefschetz fibrations for selected genera}
\end{table}

\textit{Proof of the Main Theorem}.
The proof proceeds by combining low-slope and high-slope Lefschetz fibrations using fiber sum operations, and then solving a Diophantine equation to the exact equality.

\noindent
For large enough $g$, consider a sequence of genus-$g$ Lefschetz fibrations $f_{L_n}:X_{L_n}\to \mathbb{S}^2$ whose slopes $\lambda_{L_n}$ decrease and converge to $2$, and a sequence of genus-$g$ Lefschetz fibrations $f_{U_n}:X_{U_n}\to \mathbb{S}^2$ whose slopes $\lambda_{U_n}$ increase and converge to $8$. Let $r\in(2,8)\cap\mathbb{Q}$ be an arbitrary rational number.  In particular, we can choose a rational slope \( \lambda_L \) such that
\[
2 < \lambda_L < r,
\]
and a corresponding fibration \( f_L: X_L \to \mathbb{S}^2 \) with
\[
\lambda_L = \frac{K_L^2}{\chi_L}, \quad K_L^2, \chi_L \in \mathbb{Z}.
\]

Similarly, we can choose a rational slope \( \lambda_U \) such that
\[
r < \lambda_U < 8,
\]
and a corresponding fibration \( f_U: X_U \to \mathbb{S}^2 \) with
\[
\lambda_U = \frac{K_U^2}{\chi_U}, \quad K_U^2, \chi_U \in \mathbb{Z}.
\]

Let \( f_{k,l} \) denote the fiber sum of \( k \) copies of \( f_L \) and \( l \) copies of \( f_U \). By Lemma~\ref{lemmafibersum}, the slope is
\[
\lambda_{k,l} = \frac{k K_L^2 + l K_U^2}{k \chi_L + l \chi_U}.
\]
We require \( \lambda_{k,l} = r \), which gives
\[
k K_L^2 + l K_U^2 = r (k \chi_L + l \chi_U).
\]
It can be rewritten as
\[
k (K_L^2 - r \chi_L) + l (K_U^2 - r \chi_U) = 0.
\]
Since \( \lambda_L < r < \lambda_U \), we have
\[
K_L^2 - r \chi_L < 0, \qquad K_U^2 - r \chi_U > 0.
\]
Define
\[
A := -(K_L^2 - r \chi_L) > 0, \qquad B := K_U^2 - r \chi_U > 0.
\]
Then the equation turns out to be
\[
- k A + l B = 0 \textrm{ and so} \quad \frac{k}{l} = \frac{B}{A},
\]
where both \( A \) and \( B \) are positive rational numbers. With a common denominator,
\[
A = \frac{a}{d}, \quad B = \frac{b}{d}, \qquad \textrm{for some } a,b,d \in \mathbb{N}.
\]
Then
\[
\frac{k}{l} = \frac{b/d}{a/d} = \frac{b}{a},
\]
which yields a solution by choosing \( k = b \) and \( l = a \) (both positive integers). Therefore, the fiber sum \( f_{b,a} \) is genus-\( g \) Lefschetz fibration over \( \mathbb{S}^2 \) satisfies
\[
\lambda_{f_{b,a}} = \lambda_{f_{k,l}}=r.
\]
This completes the proof.

\section{Final Remarks and Future Plans}
We are able to construct Lefschetz fibrations with slope $\lambda_f$ more close to $8$ by considering the genus $g=8m+1\geq 9 (m\geq 1 )$ Lefschetz fibrations with $\sigma(X)=0$ obtained by Baykur and Hamada \cite{baykurhamada} and the ones we have constructed in Corollary~\ref{corhigh} with $\sigma(X)< 0$. Taking fiber sum of these, we have genus $g=8m+1\geq 9 (m\geq 1 )$ Lefschetz fibrations with negative signature and slope more close to $8$.
For future direction, we aim to extend this result for any genus-$g$ Lefschetz fibrations. Secondly, we hope to obtain Lefschetz fibrations with slope $8 \leq \lambda_f < 10$ with nonnegative signature. Lastly, we study on decreasing the genus for all these type of Lefschetz fibrations.

\textbf{Acknowledgements}
The second author of this paper was supported by Scientific and Technological Research Council of T\"urkiye (T\"UB\.{I}TAK) under Grant Number 124F502. The authors thank T\"UB\.{I}TAK for their support.


\begin{thebibliography}{1}
 \bibitem{akhmedov2020}A. Akhmedov and L. Katzarkov,  \textit{Symplectic surgeries along certain singularities and new Lefschetz fibrations}, Advances in Mathematics,Vol. 360 (2020) 106920.

 \bibitem{amaros}J. Amaros, F.Bogomolov, L.Katzarkov and T.Pantev,  \textit{Symplectic Lefschetz fibrations with arbitrary fundamental groups}, J.Differential Geom. 54 (2000) 489-545; with an appendix by Ivan Smith.

 \bibitem{baykurhamada} R.I.Baykur and N. Hamada,  \textit{Lefschetz fibrations with arbitrary signature}, J. Eur. Math. Soc. (JEMS) 26 (2024), no. 8, 2837–2895.
 
 \bibitem{cadavid} C. Cadavid, \textit{A remarkable set of words in the mapping class group}, Ph.D
dissertation, Univ. of Texas at Austin (1998)

 \bibitem{Cengel2020}A. \c{C}engel and \c{C}. Karakurt, \textit{Partial fiber sum decompositions and signatures of Lefschetz
fibrations}, Topology Appl. 270 (2020), Article ID: 106937, 17 pp.

 \bibitem{LowslopeLefschetz}
A. \c{C}engel and M. Korkmaz,
\textit{Low-slope Lefschetz fibrations},
Journal of Topology and Analysis, Vol. 15, No. 2 (2023) 513--526.


\bibitem{endo}H.Endo, \textit{Meyer's signature cocycle and hyperelliptic fibrations}, Math. Ann. 316:2 (2000) 237-257.

\bibitem{endonagami}H.Endo and S.Nagami, \textit{Signature of relations in mapping class groups and nonholomorphic Lefschetz fibrations}, Trans. Amer. Math. Soc. 357 (2005) 3179-3199.

\bibitem{Gompf}R.E.Gompf and A. I. Stipsicz, \textit{4-manifolds and Kirby Calculus}, Graduate Studies in Mathematics, Vol. 20 (Amer.Math.Soc.,1999).

\bibitem{monden3}  N. Hamada, R. Kobayashi, and N. Monden, \textit{Nonholomorphic Lefschetz fibrations with (-1)- sections}, Pacific J. Math. 298:2 (2019), 375–398.

\bibitem{kas}A.Kas, \textit{On the handlebody decomposition associated to a Lefschetz fibration}, Pacific J. Math. 89:1 (1980), 89104.

\bibitem{korkmaz2001} M. Korkmaz, \textit{Noncomplex smooth $4$-manifolds with Lefschetz fibrations}. Int.
Math. Res. Notices (3),  (2001), 115–-128

\bibitem{matsumoto1} Y.Matsumoto, \textit{On 4-manifolds fibered by tori II},  Proc. Japan Acad. Ser. A. Math. Sci., 59:3 (1983), 100-103.

\bibitem{matsumoto2} Y.Matsumoto, \textit{ Lefschetz fibrations of genus two:a topological approach}, pp.123-148 in  Topology and Teichmüller spaces,(Katinkula,1995),edited by S.Kojima et al., World Scientific,1996.
 
\bibitem{miyachishiga} H. Miyachi and H.Shiga, \textit{Holonomies and the slope inequality of Lefschetz fibrations}, Proc. Amer. Math. Soc. 139:4 (2011) 1299-1307.
 
\bibitem{monden} N.Monden, \textit{Lefschetz fibrations with small slope}, Pacific J. Math., 267 (2014) 243-256.

\bibitem{monden2} N.Monden, \textit{On the geography of Lefschetz fibrations}, Handbook of group actions. Vol. II, Adv. Lect. Math. (ALM), vol. 32, Int. Press, Somerville, MA, 2015, pp. 297–330.


\bibitem{ozbagci} B.Ozbagci, \textit{Signatures of Lefschetz fibrations}, Pacific J. Math. 202:1 (2002), 99–118.

\bibitem{smith}I.Smith, \textit{Lefschetz fibrations and the Hodge bundle}, Geom. Topol. 3 (1999) 211-233.


\bibitem{stipsicz} A. I. Stipsicz, \textit{On the number of vanishing cycles in Lefschetz fibrations}, Math. Res. Lett. 6:3–4 (1999), 449–456.


\bibitem{xiao} G.Xiao, \textit{Fibered algebraic surfaces with low slope},  Math. Ann., 276 (1987) 449-466.

\end{thebibliography}
\end{document}